\documentclass[a4paper,10pt]{article}
\usepackage[plainpages=false]{hyperref}
\usepackage{amsfonts,latexsym,rawfonts,amsmath,amssymb,amsthm, mathrsfs, lscape}
\usepackage{verbatim}

\usepackage[all]{xy}
\usepackage{graphicx,psfrag}

\usepackage{array, tabularx}

\usepackage{setspace}
\usepackage[english]{babel}
\usepackage{color}
\usepackage{enumerate}

\usepackage{authblk}

\newtheorem{thm}{Theorem}[section]
\newtheorem{cor}[thm]{Corollary}
\newtheorem{lem}[thm]{Lemma}

\newtheorem{prop}[thm]{Proposition}
\newtheorem{prob}[thm]{Question}
\newtheorem{conj}[thm]{Conjecture}

\newtheorem{rmk}[thm]{Remark}

\theoremstyle{definition}

\numberwithin{equation}{section}

\def \C {\mathbb C}
\def \Z {\mathbb Z}
\def \R {\mathbb R}

\def \F {\mathcal F}

\def \T {T}

 \def \Q {\mathbb Q}
\def \P {\mathbb P}

\def \L {\mathcal L}

\def \t {\mathfrak t}

\def \p {\partial}
\def \bp {\bar{\partial}}

\def \Hilb {\textbf{Hilb}}
\def \Aut {\text{Aut}}

\begin{document}

\title{K\"ahler-Ricci flow, K\"ahler-Einstein metric, and K-stability}
\author{Xiuxiong Chen, Song Sun, Bing Wang \footnote{X.X. Chen is partially supported by NSF grant DMS-1515795;
S. Sun is partially supported by by NSF grant DMS-1405832  and Alfred P. Sloan fellowship; B. Wang is partially supported by  NSF grant DMS-1510401.}}

\maketitle{}

\begin{abstract}
We prove  the existence of K\"ahler-Einstein metric on a K-stable Fano manifold using the recent compactness result on K\"ahler-Ricci flows. The key ingredient is an algebro-geometric description of the asymptotic behavior of K\"ahler-Ricci flow on Fano manifolds. This is in turn based on a general finite dimensional discussion,  which is interesting in its own and could potentially apply to other problems. As one application, we relate the asymptotics of the Calabi flow on a polarized K\"ahler manifold to K-stability assuming bounds on geometry.
\end{abstract}

\tableofcontents

\section{Introduction}
Let $X$ be an $n$-dimensional Fano manifold. It was first conjectured by Yau \cite{Yau} that the existence of a K\"ahler-Einstein metric on $X$ is equivalent to certain algebro-geometric stability of $X$. In 2012, this conjecture was proved
by Chen-Donaldson-Sun \cite{CDS1, CDS2, CDS3}. The precise notion of stability is the so-called K-stability, defined by Tian \cite{Tian} and Donaldson \cite{Do02}. The proof depends on a deformation method involving K\"ahler-Einstein metrics with cone singularities, which was introduced by Donaldson \cite{Do11} in 2011. 

There are also other approaches to study the existence problem of K\"ahler-Einstein metrics on Fano manifolds. In general to make these work two key ingredients are needed, namely the \emph{partial-$C^0$-estimate} and the construction of a  \emph{de-stabilizing test configuration}. The first is analytic and the second is algebraic in nature. For the partial-$C^0$-estimate, it is proved by Sz\'{e}kelyhidi \cite{Sz2} for the classical Aubin-Yau continuity path, by adapting the results of \cite{DS1, CDS2, CDS3}; for the approach using Ricci flow, this is proved by  Chen-Wang \cite{CW1} in dimension two, Tian-Zhang \cite{TZhang} in dimension three, and by Chen-Wang \cite{CW2} in all dimensions as a consequence of the resolution of the Hamilton-Tian conjecture.
We note that these results together with the work of S. Paul \cite{Paul1, Paul2, Paul3} already imply that on a Fano manifold without non-trivial holomorphic vector fields,  the existence of a K\"ahler-Einstein metric is equivalent to the notion of stability defined by Paul. 

About the second ingredient, very recently Datar and Sz\'{e}kelyhidi \cite{DaSz} have adapted the results of \cite{CDS3}  to the Aubin-Yau continuity path, which gives a new proof of the theorem of Chen-Donaldson-Sun. Our focus in this paper is to give yet another proof using the Ricci flow, which means that technically we will address the issue of constructing a de-stabilizing test configuration.  Notice this can not be naively adapted from \cite{CDS3} and requires new strategy to understand the relation between the asymptotic behavior of the K\"ahler-Ricci flow and algebraic geometry. Our argument is motivated by \cite{DS2} which studies tangent cones of non-collapsed K\"ahler-Einstein limit spaces.

We now recall the basic set-up. Let $X$ be a Fano manifold, and $\omega(0)$ be a smooth K\"ahler metric in $2\pi c_1(X)$. The normalized K\"ahler-Ricci flow equation has the form 

\begin{equation} \label{eqn1-1}
\frac{\partial}{\p t} \omega(t)=\omega(t)-Ric(\omega(t)). 
\end{equation}
It is well-known that for any smooth $\omega(0)$, (\ref{eqn1-1}) has a smooth solution $\omega(t)$ for all $t\in [0, \infty)$ and the fundamental question is to understand what happens as $t$ tends to infinity. As a consequence of the Hamilton-Tian conjecture proved in \cite{CW2}, we have

\begin{thm} [Chen-Wang \cite{CW2}]  \label{thm1-2}
As $t\rightarrow\infty$,  a sequential Gromov-Hausdorff limit of $(X, \omega(t))$  is naturally a $\Q$-Fano variety endowed with a weak K\"ahler-Ricci soliton. 
\end{thm}

Actually, the convergence happens in smooth topology away from the singularities of the limit variety.  The precise statement can be found in~\cite{CW2}. We will also summarize in Section~\ref{KRF} the input that we will need from \cite{CW2}, which we emphasize is indeed a stronger result than Theorem \ref{thm1-2}, concerning local \emph{uniform} convergence of K\"ahler-Ricci flow. This says that for any fixed $T>0$, as $t\rightarrow\infty$, the flow over the interval $[t-T, t+T]$ converges (by passing to subsequences) \emph{naturally} to a limit flow, which is induced by a K\"ahler-Ricci soliton (hence is self-similar). This is a crucial ingredient for us, in analogy with the important fact used in \cite{DS2} that when we rescale the metric at a fixed point, we always get the same tangent cone if we use two equivalent rescaling sequences (see Lemma 3.1 in \cite{DS2}).

The main result we shall prove in this article is 

\begin{thm} \label{thm1-3}
There is a unique Gromov-Hausdorff limit $Z$ of $(X, \omega(t))$, as a $\Q$-Fano variety endowed with a weak K\"ahler-Ricci soliton. Moreover, if $X$ is K-stable, then $Z$ is isomorphic to $X$ endowed with a smooth K\"ahler-Einstein metric. 
In particular,  $X$ admits a K\"ahler-Einstein metric if it is K-stable.
\end{thm}

If we assume one of the Gromov-Hausdorff limit is smooth, then the uniqueness statement follows from the main result of \cite{SW} (based on the Lojasiewicz-Simon technique). The part on the relation with K-stability is new even if we assume the curvature of $\omega(t)$ is uniformly bounded, under which Sz\'{e}kelyhidi \cite{Sz1} and Tosatti \cite{Tosatti} have obtained some partial results with extra assumptions.  

Theorem \ref{thm1-3} follows from a finite dimensional result that we will elaborate in Section~\ref{sec:fdim}. The proof has its own interest.  Notice when $X$ does not admit a K\"ahler-Einstein metric, in \cite{CDS3} (hence also in \cite{DaSz}) the de-stabilizing test configuration is constructed \emph{abstractly} using the theory of Luna slices, so is in general not canonical. In our proof of Theorem \ref{thm1-3}, when $X$ is K-unstable, we will construct a  de-stabilizing test configuration (to be more precise, a filtration) \emph{naturally} out of the K\"ahler-Ricci flow $\omega(t)$.  We expect this is the unique optimal degeneration in an appropriate sense, and we will discuss this further in Section 3. 

We believe that the strategy developed here should apply to a wider class of problems.  
As an example, we will discuss an application  to the Calabi flow in Section 4.

\

As a direct corollary of the proof of Theorem \ref{thm1-3}, we also obtain a corresponding result for K\"ahler-Ricci solitons. This has been recently proved by Datar-Szekelyhidi \cite{DaSz}, using the classical continuity path.  Again the K\"ahler-Ricci flow proof seems more intrinsic. 

\begin{cor} \label{cor1-5}
Given a holomorphic vector field $V_X$ on $X$, then $(X, V_X)$ admits a K\"ahler-Ricci soliton if and only if it is relatively K-stable in the sense of Berman-Witt-Nystr\"{o}m \cite{BW}. 
\end{cor}

The ``only if" direction is proved in \cite{BW}, and the ``if" direction is a conjecture in \cite{BW}.

\section{Finite dimensional results}
\label{sec:fdim}

The results of this section are motivated by \cite{DS2}.  The discussion here focuses on a finite dimensional problem, so is technically simpler than the situation studied in \cite{DS2}. For this reason we write down a self-contained argument. 

\

 Fix a finite dimensional Hermitian vector space $E$. We denote $K=U(E)$ and $G=GL(E)$. Let $\underline{A}=\{A_i\}$ be a sequence of elements in $G$ with $A_0=\text{Id}$. Let $B_i=A_iA_{i-1}^{-1}$.  We assume there is an element  $\Lambda$  in $\sqrt{-1}Lie(K)$ with spectrum $\mathcal S\subset \R$,  such that the following holds: 

\

\textbf{$(*)$}: For any subsequence $\{\alpha\}\subset \{i\}$, passing to a further subsequence,  $B_{\alpha+1}$ and $B_{\alpha+2}$ both converge uniformly to a limit $g e^{\Lambda }g^{-1}$ for some $g\in K$.

\

\noindent Notice the element $g\in K$ is determined by the subsequence up to right multiplication by an element in $K_\Lambda=\{h\in K| h\Lambda h^{-1}=\Lambda\}$. 
It is easy to see $(*)$ is equivalent to 

\

\textbf{$(*)'$}: There is a sequence $g_i\in K$ with $g_0=\text{Id}$, such that 
$$\lim_{i\rightarrow\infty}B_{i} g_i e^{-\Lambda}g_i^{-1}=\lim_{i\rightarrow\infty} g_{i-1}^{-1}g_i=\text{Id}.$$

\

\noindent Actually, by the compactness of $K$, the two identities in \textbf{$(*)'$}  imply \textbf{$(*)$}.   On the other hand,  suppose \textbf{$(*)$} holds, then we can find a sequence $g_i\in K$ which satisfies \textbf{$(*)'$} 
except the last condition $\displaystyle \lim_{i\rightarrow\infty} g_{i-1}^{-1}g_i=\text{Id}$ is replaced by that $g_{i-1}^{-1}g_i$ converges by sequence to elements in $K_\Lambda$. Then we can simply change $g_i$ to $g_iu_i$ for appropriate choice of a sequence  $u_i\in K_\Lambda$.

\

A special case is when $A(i)=\exp(i\Lambda)$. This  corresponds to a sequence of points on a  geodesic ray in the symmetric space $G/K$. Condition $(*)$ roughly means that we are close to this special case in a certain way.  The motivation for condition $(*)$ will be seen more clearly in the next section, where we evolve geometric structures naturally towards limits that are self-similar.  
In general  we have to allow the ``gauge transformation" $g$, and for different subsequences we may get different $g$.  

\

Let  $V$ be a complex representation of $G$, and $\mathcal S(V)$ the spectrum of the $\Lambda$ action on $V$.  Our goal is to understand the limit set of $\{[A_i. v]\}$ in $\P(V)$, for a non-zero vector $v$. The main results are Proposition \ref{prop2-8}, \ref{prop2-9} and Theorem \ref{thm2-12}. 

In our discussion below, we will choose a $K$-invariant metric  on $V$, but it is not hard to see in the end the results are independent of the particular choice of  the metric.

\begin{lem}\label{lem2-0}
For any  $v\in V\setminus\{0\}$, the following holds 
$$\log |e^{2\Lambda }. v|+\log |v|\geq 2\log |e^{\Lambda }.v|, $$ and the equality holds if and only if $v$ is an eigenvector of $\Lambda $.
\end{lem}
This follows from the convexity of the function $\log |e^{t\Lambda }.v|$. This elementary result  is the key to the following discussion.

\

Denote $f_i(v)=\log |A_i. v|$. Then we have

\begin{lem} \label{lem2-1}
Given $\mu\notin \mathcal S(V)$, there is an $I=I(\mu)$, such that for all $v\neq 0$ and $j>i\geq I$, if  $f_{i+1}(v)\geq f_{i}(v)+\mu$, then  $f_{j+1}(v)>f_{j}(v)+\mu$. 
\end{lem}

\begin{proof}
Suppose this is not true, then we may find  subsequences $\{\alpha\}$  tending to infinity, and $v_\alpha\in V$, such that 
\begin{align*}
 f_{\alpha+1}(v_\alpha)\geq f_\alpha(v_\alpha)+\mu, \quad f_{\alpha+2}(v_\alpha)\leq f_{\alpha+1}(v_\alpha)+\mu. 
\end{align*}
Note that we have adjusted the sequence to obtain the above inequalities. 
Without loss of generality we may normalize $v_\alpha$ such that $|A_\alpha. v_\alpha|=1$. 
By $(*)$ and by passing to a subsequence, we may assume $A_{\alpha+k}A_{\alpha}^{-1}$ converges $ge^{k \Lambda }g^{-1}$ for $k=0, 1, 2$ and some $g\in K$. 
Passing to a further subsequence we may also assume $A_{\alpha+k}. v_\alpha$ converges to $w_k$,  with $w_k=ge^{k\Lambda }g^{-1}. w_0$. 
Let $w=g^{-1}. w_0$, then by our assumption  we have 
\begin{align*}
  \log |e^{2\Lambda }. w|-\log |e^{\Lambda }.w|\leq \mu \leq \log |e^\Lambda . w|-\log |w|. 
\end{align*}
By Lemma \ref{lem2-0} we conclude $w$ is an eigenvector of $\Lambda $, with eigenvalue $\mu$. This contradicts our choice of $\mu$.
\end{proof}

\begin{lem}\label{lem2-2}
$\displaystyle d(v)=\lim_{i\rightarrow\infty} i^{-1}f_i(v)$ is well-defined and belongs to $\mathcal S(V)$. 
 \end{lem}

 \begin{proof} 
It follows from Lemma~\ref{lem2-1} that  the limit 
\begin{align}
\tilde d(v)=\lim_{i\rightarrow\infty}(f_{i+1}(v)-f_{i}(v))    \label{eqn:dv}
\end{align}
exists and $\tilde d(v)\in \mathcal S(V)$. It is then an elementary exercise to show that 
\begin{align}
  d(v)=\tilde d(v).   \label{eqn:ddv}
\end{align}
 \end{proof} 
 
 \begin{rmk}
If $A_i=\exp(i\Lambda)$ then $d(v)$ is the well-known weight function associated to a geodesic ray, as the slope at infinity of the Kempf-Ness function.  
 \end{rmk} 
 
 For $v\in V\setminus\{0\}$, we denote by $[v]$ the corresponding point in $\P(V)$. Let $U$ be the eigenspace of $\Lambda$ with eigenvalue $d(v)$. 
 
 \begin{lem} \label{lem2-3}
Given any subsequence $\{\alpha\}$ so that $B_{\alpha+1}$ converges to $ge^{\Lambda }g^{-1}$,   passing to a further subsequence $\{\beta\}$, $[A_\beta. v]$ converges to a limit $[w]$ with $g^{-1}. w\in U$. 
 \end{lem}
 
 \begin{proof}
 From the above discussion, we have
 $$\lim_{\alpha\rightarrow\infty}(f_{\alpha+2}(v)-f_{\alpha+1}(v))=\lim_{\alpha\rightarrow\infty} (f_{\alpha+1}(v)-f_\alpha(v))=d(v).$$
  Then the conclusion follows from the proof of Lemma \ref{lem2-1}.   
 \end{proof}
 
 In terms of condition $(*)'$, this means that $[g_{\beta}^{-1}A_\beta. v]$ converges to $[g^{-1}. w]\in \P(U)$.  Then we define
the limit set $Lim(v)$ to be the union of the $K_\Lambda$-orbits of all possible sequential limits of $[g_{i}^{-1}A_i. v]$.  By the above lemma $Lim(v)$ is a subset of $\P(U)$, and it is independent of the choice of $g_i$ in $(*)'$. 
Notice by definition any sequential limit of $[A_i. v]$ is in the $K$-orbit of some element in $Lim(v)$.

  \begin{lem} \label{lem2-4}
  $Lim(v)$ is compact and connected. 
  \end{lem} 
  
  \begin{proof}
  The compactness is clear. Suppose it is not connected, then we may write $Lim(v)=O_1\cup O_2$, with $O_1, O_2$ compact and $K_\Lambda$-invariant,  and 
  \begin{align*}
    d_{FS}(O_1, O_2)\geq\epsilon>0, 
  \end{align*}
  where $d_{FS}$ denotes the Fubini-Study metric on $\P(V)$.
   Note that $K_{\Lambda}$ is connected.  By definition we can find a subsequence $\{\alpha\}\subset\{i\}$,
   such that   
   \begin{align}
     d([g_\alpha^{-1}A_{\alpha}. v], O_1)\leq\epsilon/2, \quad d([g_{\alpha+1}^{-1}A_{\alpha+1}. v], O_1)>\epsilon/2.   \label{eqn:connect}
   \end{align}
   Passing to a subsequence we may assume $[g_\alpha^{-1} A_{\alpha}.v]$ converges to $[w]$.  By our choice of $\{\alpha\}$ we know $[w] \notin  O_2$, so $[w] \in  O_1$. 
   By $(*)'$ we know $g_{\alpha+1}^{-1}B_{\alpha+1}g_{\alpha}$ converges to $e^\Lambda$ so $[g_{\alpha+1}^{-1}A_{\alpha+1}.v]$ also converges to $[w]\in O_1$. 
   This contradicts the second inequality of (\ref{eqn:connect}). 
  \end{proof}

  Now we assume $V=E$, the standard representation of $G$. We list the elements in $\mathcal S$ in decreasing order as $\lambda_1>\lambda_2>\cdots>\lambda_r$.  Then we have an orthogonal decomposition 
 $$E=\bigoplus_{s=1}^r U_s, $$
 where $U_s$ is the eigenspace associated to the eigenvalue $\lambda_s$. We also denote $n_s=\dim U_s$. 
 The sequence $\underline{A}$ defines a filtration:
 \begin{equation} \label{eqn2-1}
 E=V_1\supset \cdots \supset V_r\supset V_{r+1}=\{0\}
 \end{equation}
where $V_s$ consists of vectors $v$ with $d(v)\leq \lambda_s$, and we make the convention that $d(0)=-\infty$.

For any $p\leq m$, recall we have the Pl\"ucker embedding of the Grassmannian $G(p; E)$ into $\P(\bigwedge^pE)$ as a closed subvariety. 
Given a $p$ dimensional subspace $W\subset E$, we choose an element $\widehat W\in \bigwedge^p E$ representing $W$. Then we apply the above discussion to $V=\bigwedge^pE$, and define $d(W):=d(\widehat W)$. This is independent of the particular choice of $\widehat W$. For simplicity of notation, we will simply denote $\widehat W$ also by $W$, and the meaning will be clear from the context. For example, when we say a sequence $[W_i]$ converges to $[W_\infty]$, we mean that the corresponding $[\widehat W_i]$ converges to $[\widehat W_\infty]$ in $\P(\bigwedge^p E)$. This is also equivalent to saying that the corresponding sequence converges in $G(p; E)$.

 \begin{prop} \label{prop2-5}
 For all $s$ we may find a subspace $W_{s}\subset E$ such that $V_s=W_{s}\bigoplus V_{s+1}$, and  $Lim( W_s)=[U_s]$. 
 \end{prop}
 
This is a consequence of the following two lemmata. For $s=1, \cdots, r$, 
we define the following numbers for the simplicity of notations. 
\begin{align*}
p_s \triangleq \sum_{k\leq s} n_k, \quad q_s \triangleq \sum_{k\geq s} n_k, \quad \mu_s \triangleq \sum_{k\leq s} n_k \lambda_k, \quad \nu_s \triangleq \sum_{k\geq s} n_k \lambda_k. 
\end{align*}

 \begin{lem} \label{lem2-6}
 For all $s$, there is a subspace $R_s\subset E$ of dimension $p_s$, such that $ Lim(R_s)=[\bigoplus_{k\leq s} U_k]$. 
 \end{lem}
 
 \begin{proof}
Fix a subsequence $\{\alpha\}$ such that $B_{\alpha+1}$ converges uniformly to $ge^{\Lambda }g^{-1}$ for some $g\in K$. Fix any number $\mu\in (\mu_s+\lambda_{s+1}-\lambda_s, \mu_s)$. Notice for any $s$, we have by definition $\log |e^{\Lambda}. U_s|=n_s\lambda_s \log |U_s|$. If we let $R_s=A_\alpha^{-1}g. \bigoplus_{k\leq s} U_k$ for $\alpha$ large, then $f_{\alpha+1}(R_s)\geq f_{\alpha}(R_s)+\mu$. So by Lemma \ref{lem2-1} if we choose $\alpha>I(\mu)$, then indeed we obtain $d(R_s)>\mu$. It is easy to see that $\bigwedge^{p_s}(\bigoplus_{k\leq s} U_k)\subset \bigwedge^{p_s}\C^m$ is the unique one dimensional eigenspace of $\Lambda$ with eigenvalue bigger than  $\mu$.  
By  Lemma \ref{lem2-3}, we know $d(R_s)=\mu_s$ and $Lim (R_s)=[\bigoplus_{k\leq s} U_k]$.
 \end{proof}
 
\begin{lem} \label{lem2-7}
  For all $s$, there is a subspace $Q_s\subset E$ of dimension $q_s$, such that $Lim(Q_s)=[\bigoplus_{k\geq s} U_k]$. 
\end{lem}
 
\begin{proof}
We use the same subsequence $\{\alpha\}$ as in the proof of previous lemma, and let $\mu\in (\nu_s, \nu_s+\lambda_{s-1}-\lambda_s)$. We define $S_\alpha=A_\alpha^{-1}g. \bigoplus_{k\geq s}U_k$,
then for $\alpha$ large we have $f_{\alpha+1}(S_\alpha)\leq f_{\alpha}(S_\alpha)+\mu$. Now apply Lemma \ref{lem2-1} reversely, we see that for all  $i\in [I(\mu), \alpha)$, we have
\begin{align}
 f_{i+1}(S_\alpha)\leq f_{i}(S_\alpha)+\mu.   \label{eqn:dddv}
\end{align}
Let $\alpha$ tend to infinity we may pass to a subsequence and assume $[S_\alpha]$ converges to a limit, which we denote by $[Q_s]$. 
Combining (\ref{eqn:dv}), (\ref{eqn:ddv}) and (\ref{eqn:dddv}), we obtain $d(Q_s)\leq \mu$. Hence $d(Q_s)=\nu_s$, by an argument similar to the proof of Lemma~\ref{lem2-6}.
Then we apply Lemma~\ref{lem2-3} to obtain $Lim(Q_s)=[\bigoplus_{k\geq s}U_k]$.
\end{proof}
   
\begin{proof}[Proof of Proposition \ref{prop2-5}]
It follows easily from Lemma \ref{lem2-6} and Lemma \ref{lem2-7} that for all $s$, $R_s\cap V_{s+1}=0$, and $Q_s\subset V_s$, by their definitions.
So we have 
\begin{align*}
q_s=\dim Q_s\leq \dim V_s\leq \sum_{s} n_s\lambda_s-\dim R_{s-1}=\sum_{s}n_s\lambda_s-\mu_{s-1}=q_{s}.
\end{align*}
Therefore, all the inequalities in the above line become equalities. In particular, we have $V_s=Q_s$. 
Define $W_s=R_s\cap Q_s$. 
Then we have
\begin{align*}
  Lim(W_s)=Lim(R_s)\cap Lim(Q_s)=[U_s]. 
\end{align*}
It follows from definitions and the above equalities that $W_s\subset V_s$, $W_s\cap V_{s+1}=0$.   Moreover,  the above equalities imply that $\dim W_s =n_s$ and consequently
\begin{align*}
  \dim W_s+\dim V_{s+1}=\dim V_s. 
\end{align*}
This yields that $V_s=W_s\bigoplus V_{s+1}$. 
\end{proof}
 
Now fix a choice of $W_s$  in Proposition \ref{prop2-5}.  Then we can define a (real) one-parameter subgroup $\lambda(t)=\exp(t\xi)$ of $G$, where $\xi$ acts on $W_s$ by multiplication by $\lambda_s$. 
Since $Lim(W_s)=[U_s]$, we may find a sequence $C_i\in G$ such that $C_i$ converges uniformly to the identity, and $C_i$ identifies $g_i^{-1}A_i. W_{s}$ with $U_s$ for all $s$. Let $\widetilde A_i=C_ig_i^{-1}A_i C_0^{-1}$, and $\widetilde B_i=\widetilde A_i {\widetilde A_{i-1}}^{-1}$.  From the construction, $\widetilde A_i, \widetilde B_i\in G_\Lambda$, where $G_\Lambda=\{g\in G|g\Lambda g^{-1}=\Lambda\}$, and by our choice of $g_i$ we have  
\begin{equation} \label{eqn2-2}
\lim_{i\rightarrow\infty} \widetilde B_i=e^{\Lambda}. 
\end{equation}
Moreover,  $\lambda(t)=C_0^{-1}e^{t\Lambda} C_0$. 
 
\

Now we return to a general representation $V$. Let $V=\bigoplus_{j=1}^e \mathcal U_j$, where $\mathcal U_j$ is the eigenspace of $\Lambda $ associated to the eigenvalue $\tau_j$, and $\tau_j$ is arranged in a decreasing order. We also have the filtration 

$$V=\mathcal V_1\supset \cdots \supset \mathcal V_e\supset \mathcal V_{e+1}=\{0\}, $$
where $\mathcal V_j$ consists of elements $v$ with $d(v)\leq \tau_j$. 
Given $[v]\in \P(V)$, we denote $\displaystyle [\bar v]=\lim_{t\rightarrow\infty} C_0 \lambda(t). [v]\in \P(V)$. Then $[\bar v]$ is fixed by $\Lambda$.   

\begin{prop} \label{prop2-8}
Any point $[w]\in Lim(v)$ is in the closure of  the $G_\Lambda$-orbit of $[\bar v]$. 
\end{prop}

\begin{proof}
Suppose for some subsequence $\{\alpha\}$,  $g_\alpha^{-1}A_\alpha. [v]$ converges to $[w]$. Then 
\begin{align*}
  [v_\alpha] \triangleq [\widetilde A_\alpha C_0. v]=[C_\alpha g_\alpha^{-1} A_\alpha. v]
\end{align*}
also converges to $[w]$. 
Suppose $[w]\in \P(\mathcal U_j)$, then $C_0.v \in \bigoplus_{k\geq j} \mathcal U_k$ by (\ref{eqn2-2}). 
Therefore $[\bar v]$ is the projection of $[C_0.v]$ to $\P(\mathcal U_j)$. Since $\widetilde A_\alpha\in G_\Lambda$, $\widetilde A_\alpha. [\bar v]$ is the projection of $[v_\alpha]$ to $\P(\mathcal U_j)$. Since the projection map to $\P(\mathcal U_j)$ is continuous in a neighborhood of $[w]$, it follows that 
$\displaystyle \lim_{\alpha\rightarrow\infty}\widetilde A_\alpha. [\bar v]=[w]$. 
\end{proof}

In general the above $\lambda(t)$ depends on the choice of $W_s$. Let $P(\xi)$ be the parabolic subgroup of $G$ consisting of elements $p$ such that 
$\displaystyle \lim_{t\rightarrow\infty} \lambda(t)p\lambda(t)^{-1}$ exists. If we are given another choice of complementary subspaces $W_s'$, then we have $\xi'=p\xi p^{-1}$ for some $p\in P(\xi)$. In particular,  the conjugacy class of $\xi$ under the action of  $P(\xi)$ is  independent of the particular choice of $W_s$, so is uniquely determined by the filtration (\ref{eqn2-1}). This equivalence relation is well-studied in geometric invariant theory, see for example \cite{GRS}, Appendix C. We also have

\begin{prop} \label{prop2-9}
The $G_\Lambda$-orbit of $[\bar v]$ is uniquely determined by $\underline{A}$ and $[v]$. 
\end{prop}

\begin{proof}
This is not hard to see. The point is that  the choice of $C_0$ identifying  $W_s$ with $U_s$ for all $s$, is unique up to the action of $G_\Lambda$. 
\end{proof}

Now we introduce the following property on $[v]$: 

\

\noindent  \textbf{Property (R)}: Every element $[w]\in Lim(v)$ has a reductive stabilizer group in $G_\Lambda$, and $G_\Lambda. [w]=K_\Lambda. [w]$.

\

\noindent This is often satisfied in the concrete geometric situation, as we shall see later. 

\begin{thm} \label{thm2-12} 
Suppose $v$ satisfies Property (\textbf{R}), then $Lim(v)=K_\Lambda. [v_\infty]$ for a unique element $[v_\infty]\in \P(V)$. Moreover, there is
 an algebraic one-parameter subgroup $\phi: \C^*\rightarrow G_\Lambda$ that degenerates $[\bar v]$ to $[v_\infty]$, i.e. 
 $\displaystyle \lim_{t\rightarrow0} \phi(t). [\bar v]=[v_\infty]$. 
\end{thm}

\begin{proof}
This follows from exactly the same arguments as in \cite{DS2} (the discussion before Remark 3.18). For the convenience of readers we repeat the proof here.  Let $[v_\infty]$ be a point in $Lim(v)$ whose stabilizer group in $G_\Lambda$ has minimal dimension.   By \cite{Do10} we can find an equivariant slice $\P'$ at $[v_\infty]$ for the action of $G_\Lambda$.  Let $O$ be the $G_\Lambda$ orbit of $[w]$, and $O'=O\cap \P'$. Notice by general theory the closure $\overline {O'}$ is a (possibly reducible) algebraic variety. By Proposition \ref{prop2-8} we have $[v_\infty]\in \overline{O}$. So from the construction of $\P'$ in \cite{Do10} we can find a small neighborhood $ \mathcal W$ of $[v_\infty]$ in $\P$, such that  each component of $O'\cap \mathcal W$  is contained in a single $G_\Lambda$-orbit. Moreover any point in $\overline O\cap  \mathcal W$ is in the $G_\Lambda$-orbit of a point in $\overline{O'}\cap  \mathcal W$. In particular, $[v_\infty]\in \overline{O'}$.
  
  Now we can choose an open neighborhood $N$ of $[v_\infty]$ in $Lim(v)$ such that $N\subset  \mathcal W$.  By connectedness of $Lim(v)$ it suffices to show that $N\subset K_\Lambda. [v_\infty]$.  Now for any $[v_\infty']\in N$,  by the construction of slice we may find $g\in G_\Lambda$ such that $g. [v_\infty']\in \overline{O'}\cap  \mathcal W$. Then by the classical geometric invariant theory we know  $[v_\infty]$ is in the closure of the $K_\Lambda$-orbit  of $g. [v_\infty']$. Therefore $[v_\infty]$ and $[v_\infty']$ must be in the same $G_\Lambda$-orbit, since otherwise $[v_\infty']$ would have a stabilizer group in $G_\Lambda$ with smaller dimension, which contradicts our choice of $[v_\infty]$. By Property (\textbf{R}) again we conclude $[v_\infty']\in K_\Lambda. [v_\infty]$. 
\end{proof}

In particular, this says that  there is a two step degeneration from $[v]$ to $[v_\infty]$, through $[\bar v]$. Notice in contrast to $\lambda(t)$,  the above algebraic one-parameter subgroup $\phi(t)$ is constructed using abstract theory, so is in general not canonically defined.

\

In practice there are possible variants of the above discussion. Instead of a sequence $\{A_i\}$ we often have a continuous path 
$$\{A(t), t\in [0, \infty)\}.$$
The property \textbf{$(*)$} is then replaced by 

\

\textbf{$(**)$}:  for any sequence $t_i\rightarrow\infty$, passing to a subsequence, the path 
$$[0, 2]\rightarrow G; \; t \mapsto A(t_i+t)A^{-1}(t_i)$$ 
converges uniformly to a limit $ge^{t\Lambda }g^{-1}$ for some $g\in K$. 

\

In this case one can similarly prove that 
\begin{equation} \label{eqn2-3}
d(v)=\lim_{t\rightarrow\infty} t^{-1}\log |A(t). v|
\end{equation}
is well-defined, and agrees with the definition using $A(i)$ for $i\in \Z$. Then the above discussion can be repeated, and we also have a filtration of $\C^m$ defined by $A(t)$, just as (\ref{eqn2-1}).

\

\section{Asymptotics of K\"ahler-Ricci flow}

\subsection{A general discussion}

Let $(X, L)$ be an $n$-dimensional polarized K\"ahler manifold. Let $h(t)$ be a smooth family of Hermtian metrics on $L$ with induced K\"ahler metrics $\omega(t)\in 2\pi c_1(L)$.  We define a family of Hermitian inner products $H_t$ on $H^0(X, L)$  by
\begin{equation} \label{eqn3-7}
H_t(s_1, s_2):=\int_X \langle s_1, s_2\rangle_{h(t)} \omega^n(t).
\end{equation}

We assume $L$ is very ample, i.e. the natural map $F: X\rightarrow \P(H^0(X, L)^*)$ and moreover we assume that the natural map $\iota_k: \text{Sym}^k H^0(X, L)\rightarrow H^0(X, L^k)$ is surjective for all $k\geq 1$. Notice both can be achieved by replacing $L$ with $L^a$ for sufficiently big $a$.   Let  $E$ be $H^0(X, L)^*$ endowed with the metric induced by $H_0$. 
Following the notation of Section 2 we denote $G=GL(E)$ and $K=U(E)$. The path $H_t$ determines a smooth path $\widetilde A(t)$ in $G/K$ with $\widetilde A(0)=\text{Id}$. Let $A(t)$ be the parallel lift of $\widetilde A(t)$ to $G$, with respect to the natural connection when we view $G$ as a principal $K$ bundle over $G/K$. This means that for each $t$, $\dot A(t)A(t)^{-1}$ is Hermitian symmetric with respect to the metric induced by $H_0$.

 More concretely, choosing an orthonormal basis $\{s_\alpha\}$ of $H^0(X, L)$ with respect to $H_0$, we obtain a smooth family of orthonormal basis of $\{s_\alpha(t)\}$ of $H^0(X, L)$ with respect to $H_t$, 
  by solving the ODE 
\begin{equation}  \label{eqn3-5}
\frac{\p s_\alpha(t)}{\p t}=-\frac{1}{2} \dot{H}_t(s_\alpha(t), s_\beta(t))s_\beta(t)
\end{equation}
with initial value $s_\alpha(0)=s_\alpha$. It is easy to see the linear transformation on $E$ that maps the corresponding dual basis $s^\alpha$ to $s^\alpha(t)$ is independent of the choice of $\{s_\alpha\}$, and agrees with the above $A(t)$. Indeed,  with respect to the basis $\{s^\alpha\}$ we have
$$(\dot{A}(t)A(t)^{-1})_{\alpha \beta}=\frac{1}{2}\dot{H}_t(s_\beta(t), s_\alpha(t))$$
which is Hermitian symmetric.  Notice for any $s\in E^*$, we have

\begin{equation} \label{eqn3-6}
||A(t). s||_{H_0}=||s||_{H_t}. 
\end{equation}

The path $A(t)$ generates a family of embedding $F_t: X\rightarrow \P(E)$, with $F_t=A(t)\circ F$.
Let $\Hilb$ be the Hilbert scheme parametrizing sub-schemes of $\P(E)$ with the same Hilbert polynomial as $(X, L)$. By construction it is a closed subscheme of some $\P(V)$, where $V$ is a natural representation of $G$ and $\Hilb$ is $G$-invariant. Then the above $F_t$ gives rise to a continuous path $[X_t]$ in $\Hilb$ satisfying  $[X_t]=A(t). [X_0]$. 

\

In what follows we shall assume

\

\textbf{(H1)}: There is an element $\Lambda\in \sqrt{-1}Lie(K)$ such that $\{A(t)\}$ satisfies $(**)$. 

\

 We can then apply the discussion of Section 2 to the path $\{A(t)\}$ and the representation $V$, with $[v]$ replaced by the point $[X]$ in $\Hilb\subset \P(V)$. Since $\Hilb$ is closed and $G$-invariant,  the limit set $Lim(X)$ and the point $[\overline X]$  are both contained in $\Hilb^\Lambda$, the subscheme of $\Hilb$ parametrizing $\Lambda$-invariant sub-schemes.

We also know the $G_\Lambda$-isomorphism class of $[\overline X]\in \Hilb^\Lambda$ is uniquely determined by $A(t)$.
It is interesting to understand the coordinate ring of $\overline X$  in terms of the language of filtrations introduced in  \cite{NW}. For all $k\geq 1$, $G$ acts naturally on $\text{Sym}^k R_1$, which is
$\otimes^k R_1/S_k$ with the symmetric group $S_k$ naturally acting on $\otimes^k R_1$. 
By Section 2, we see that $\{A(t)\}$ generates a filtration of $\text{Sym}^k R_1$.  Then it also induces a filtration  of $R_{k}=H^0(X, L^k)$ under the map $\iota_k: \text{Sym}^k R_1\rightarrow R_k$. More precisely, for $s\in R_k$, we define 
$$d(s) \triangleq \inf\{d(f)|f\in \text{Sym}^k R_1, \iota_k(f)=s\}, $$
where $d(f)$ is defined in Lemma \ref{lem2-2}. Notice each $H_t$ defines a natural metric on $\text{Sym}^k R_1$, so induces a metric on $R_k$
$$||s||_{H_t^*}^2=\inf\{||f||_{H_t}^2|f\in \text{Sym}^k R_1, \iota_k(f)=s\}.$$
Using (\ref{eqn3-6}) we have
\begin{equation} \label{eqn3-8}
d(s)=\lim_{t\rightarrow\infty} t^{-1}\log ||s||_{H_t^*}.
\end{equation}
  We  also make the convention that $d(s)=0$ for $s\in H^0(X, L^0)\simeq \C$.
  For all $d\in \R$, we set
 \begin{align*}
   \F_dR_k \triangleq \{s\in R_k|d(s)\leq d\},  \qquad \F_dR \triangleq \bigoplus_{k\geq 0}\F_{d}R_k.
 \end{align*}
Let $R=\bigoplus_{k\geq 0}H^0(X, L^k)=\bigoplus_{k\geq 0} R_k$ be the homogeneous coordinate ring of $(X, L)$. 
  Then $\F:=\{\F_d\}_{d\in \R}$ is an increasing filtration of subspaces of $R$.  It is multiplicative in the sense that $\F_{d}R_k\cdot \F_{e}R_l\subset \F_{d+e}R_{k+l}$. This follows from the simple fact that 
for any $f_1\in \text{Sym}^k R_1$ and $f_2\in \text{Sym}^l R_1$ we have  $d(f_1\cdot f_2)=d(f_1)+d(f_2)$.

Clearly $\F$ is only discontinuous at a discrete set of values $d\in \R$, which is contained in the sub semigroup of $\R$ generated by the spectrum of $\Lambda$-action on $E^*$. We denote these by  $\cdots<d_{i-1}<d_i\cdots$, and define the associated  ring
 $$\overline R \triangleq \bigoplus_{i} \F_{d_i}R/\F_{d_{i-1}}R. $$
It is endowed with two gradings, one by $\{k\}$ and the other by $\{d_i\}$. 

\begin{prop} \label{prop3-1}
The coordinate ring of $\overline X$ is isomorphic to $\overline R$, and the action of $\Lambda$ is encoded in the grading by $\{d_i\}$.  
\end{prop}

\begin{proof}
The discussion of Section 2 produces an element $C_0\in G$ which identifies each $W_s$ with $U_s$ in $E$ (we adopt the notation there), such that
\begin{align*}
 C_0^{-1}. [\overline X]=\lim_{t\rightarrow\infty}C_0^{-1}e^{t\Lambda}C_0. [X]. 
\end{align*}
The action of $C_0^{-1}\Lambda C_0$ defines a \emph{new} grading $\{d\}$ on the ring $\bigoplus_{k\geq 0}\text{Sym}^k R_1$. Namely, for each $k$, we have a weight decomposition 
$$\text{Sym}^k R_1=\bigoplus_{d\in \R} V_{k, d}. $$
It is easy to see from definition that for any $f\in \text{Sym}^k R_1$ with a weight decomposition $f=\sum _{d} f_d$ (this is of course always a finite sum), we have
$$d(f)=\sup\{d|f_d\neq 0\}.$$
Now we define a map 
$$\Phi: \bigoplus_{k\geq 0} \text{Sym}^k R_1\rightarrow \overline R$$
which sends an element $f\in \text{Sym}^kR_1$ to the corresponding class $[\iota_k(f)]$ in $\F_{d_i}R/\F_{d_{i-1}}R$ for $d_i=d(f)$. By the above discussion $\Phi$ is surjective. Let $I$ be the saturated ideal defining $X$, then the kernel of $\Phi$ is exactly the \emph{initial ideal} of $I$, i.e. the ideal generated by the initial terms of elements in $I$, with respect to the above new grading on $\bigoplus_{k\geq 0} \text{Sym}^k R_1$. 
We denote the initial ideal mentioned above by $\bar{I}$. Then we have $\overline R=\bigoplus_{k\geq0} \text{Sym}^k R_1/\overline I$. From the construction of the Hilbert scheme, the latter is exactly the  homogeneous coordinate ring of $C_0^{-1}. \overline X$, and the grading by $\{d_i\}$ on $\overline R$ corresponds to the action of $C_0^{-1}\Lambda C_0$ on $C_0^{-1}. \overline X$. 

 Now the conclusion follows from the fact that $(\overline X, \Lambda)$ and $(C_0^{-1}. \overline X, C_0^{-1}\Lambda C_0)$ have isomorphic graded homogeneous ring. 
 \end{proof}

 For our purpose, it is often convenient to re-grade $\F$. Let $\underline\lambda$ be a number smaller than the smallest eigenvalue of the $\Lambda$ action on $R_1$, then we define 
 \begin{align*}
   \F'_d R_k \triangleq \F_{d-\underline \lambda k}R_k, \qquad \F_d'R \triangleq \bigoplus_{k} \F'_dR_k. 
 \end{align*}
 Then the new filtration $\F'=\{\F'_dR\}$ is again multiplicative. Moreover, it is ``positive" in the sense that $\mathcal F'_{0}R=\C$. It is easy to see the  graded ring associated to $\F'$ only differs from $\overline R$ by a shift of the grading, and geometrically, it defines the same variety $\overline X$ with the same projective action of $\Lambda$, with a different choice of \emph{linearization} on $R_1$ by $\Lambda'=\Lambda-\underline\lambda Id$. 
 
We say the filtration $\F$ is \emph{rational} if  we can find $\underline \lambda$ such that $\Lambda'$ has rational spectrum. In this case, we can find a smallest integer $D$ and some $\underline \lambda$, such that $D\cdot \mathcal S(\Lambda')\subset \Z$. Then we define a new filtration $\{\F''_jR\}_{j\in \Z_{\geq0}}$by setting
 $$\F''_j R=\bigcup_{d\leq D^{-1}j}\F'_{d}R.$$
 The associated graded ring again defines the same variety $\overline X$, but the induced action has been re-scaled to $D\Lambda'$. 
 
 As in \cite{NW, Sz3} we form the Rees algebra 
$$\text{Rees}(\F'')=\bigoplus_{k\geq 0} \F''_k t^k\subset R[t]$$
which gives rise to a test configuration for $X$ with central fiber $\overline X$. Geometrically, the rationality of $\F$ means that $\Lambda$ generates an algebraic one parameter subgroup $\chi: \C^*\rightarrow G$, 
such that $\displaystyle \lim_{t\rightarrow 0} \chi(t). [X]=C_0^{-1}. [\overline X]$. 

When $\F$ is not rational, $\sqrt{-1}\Lambda$ generates a compact subtorus $T\subset K$, with rank bigger than one. Then we can still construct a test configuration in a non-canonical way. More precisely, we want to perturb $\Lambda$ within $\sqrt{-1}Lie(T)$, while keeping the associated ring $\overline R$ invariant (of course the grading will change). Notice when we vary $\Lambda$ in $\sqrt{-1}Lie(T)$, we actually change the grading on $\bigoplus_{k\geq 0}\text{Sym}^kR_1$; indeed we are \emph{weakening} the grading in that a graded piece does not split but different graded pieces can emerge. Now suppose $\overline I$ is generated by the $g_1, \cdots, g_p$, where each $g_i$ is the initial term of some $f_i\in I$ with respect to the grading defined by $C_0^{-1}\Lambda C_0$. Then it follows that  for a rational $\Gamma\in Lie(T)$ close to $\Lambda$, $g_i$ is also the initial term of $f_i$ with respect to the grading defined by $C_0^{-1}\Gamma C_0$. This implies the initial ideal  $J$ of $I$ with respect to the new grading contains $\overline I$. By the proof of Proposition \ref{prop3-1} we know for all $k\geq 1$, 
$$\dim \text{Sym}^k R_1/\overline I_k=\dim \overline R_k=\dim R_k.$$
Similarly 
$$\dim \text{Sym}^k R_1/J_k=\dim R_k.$$
This implies $J=\overline I$. It follows that we can use the filtration defined by $\Gamma$ to construct a test configuration for $X$ with central fiber $\overline X$, and the induced $\C^*$-action is generated by $\Gamma$.

\

Now we come back to the filtration $\mathcal F$. For all $k\geq 1$, we have a natural $L^2$ inner-product $H_t$ on $R_k$, defined just like (\ref{eqn3-7}).  Suppose $(X, h(t))$ satisfies an extra hypothesis

\

\textbf{(H2)}:   For all $k\geq 1$, there is a constant $C_k>0$ such that for all $t\geq 0$,  we have on $R_k$, 
\begin{align*}
 C_k^{-1} H_t\leq H_t^*\leq C_k H_t.
\end{align*}
Then by (\ref{eqn3-8}), we have  
$$\displaystyle e(s)=\lim_{t\rightarrow\infty} t^{-1}\log ||s||_{H_t}.$$
In particular, the filtration $\F$, and hence $(\overline X, \Lambda)$, is intrinsically defined by $(X, h(t))$. In other words, suppose we replace $L$ by $L^k$ for some $k\geq 1$ in the above discussion, and suppose again \textbf{(H1)} holds, then we will end up with the same filtration. 

\

\subsection{K\"ahler-Ricci flow on Fano manifolds}
\label{KRF} 

Now we prove Theorem \ref{thm1-3}, so we assume $X$ is Fano. Let $\omega(t)$ be a solution of (\ref{eqn1-1}) with $\omega_0\in 2\pi c_1(X)$.  To obtain the corresponding family of  Hermitian metrics $h(t)$ on $K_{X}^{-1}$, we use the normalization
\begin{equation} \label{eqn3-0}
\int_X \Omega_{h(t)}=\int_X \omega^n(t), 
\end{equation}
where $\Omega_{h(t)}$ is the volume form on $X$ naturally associated to $h(t)$. The corresponding K\"ahler potential $\phi(t)=-\log (h(t)h(0)^{-1})$ satisfies the usual normalized equation, 
\begin{align*}
  \dot{\phi}=\log \frac{\omega_{\phi}^n}{\omega^n} + \phi - u_{\omega},
\label{eqn:HA01_5}
\end{align*}
where $u_{\omega}$ is the Ricci potential with condition $\int_M e^{-u_{\omega}} \frac{\omega^n}{n!}=(2\pi)^n$.

We first summarize the relevant results proved in \cite{CW2}.  First of all,  as $t\rightarrow\infty$,  one can take sequential polarized Gromov-Hausdorff limits, in the sense of \cite{DS1}. Such a limit $Z$ is naturally a $\Q$-Fano variety, 
endowed with a weak K\"ahler-Ricci soliton metric $\omega_Z$, in the sense of \cite{BW}(c.f. the Remark after Proposition 4.15 of \cite{DS1} for a similar discussion).
In particular, there is a continuous Hermitian metric $h_Z$ on the $\Q$-line bundle $K_Z^{-1}$, which is smooth on the smooth locus $Z^s$ of $Z$, with curvature form 
$-\sqrt{-1} \omega_Z$. Moreover, $\omega_Z$ is a genuine K\"ahler form on $Z^s$, and  there is a holomorphic vector field $V_Z$ on $Z$, such that $JV_Z$ generates holomorphic transformations of $Z$ that preserves $\omega_Z$, and such that the equation  $Ric(\omega_Z)=\omega_Z+\L_{V_Z}\omega_Z$ holds on $Z^s$. 

Let $\mathcal C$ be the set of all such sequential limits, and $\overline{\mathcal C}$ be the union of $\mathcal C$ and  $\{X_t=(X, J, \omega(t))|t\geq 0\}$.  Then $\overline {\mathcal C}$ is endowed with the polarized Gromov-Hausdorff topology, in the sense of \cite{DS1}. It is  easy to see both $\mathcal C$ and $\overline{\mathcal C}$ are compact and connected (we refer to \cite{DS2}, Lemma 2.7 and Lemma 3.2 for a proof of similar results). 

As a consequence of the main result of \cite{CW2} and the discussion in \cite{DS1}, there are positive integers $r$ and $m$ (depending only on $(X, \omega_0)$), such that any $Z\in \overline{ \mathcal C}$ is holomorphically embedded into $\P^{m-1}$ by $L^2$ orthonormal sections of $H^0(Z, K_Z^{-r})$, and the image lies in a fixed Hilbert scheme $\Hilb$.  Moreover, the natural map $\overline{\mathcal C}\rightarrow \Hilb/U(m)$ is continuous. We may also assume that the map $\iota_k: \text{Sym}^kH^0(X, K_X^{-r})\rightarrow H^0(X, K_X^{-rk})$ is surjective for all $k\geq 1$. We are therefore in the setting of Section 3.1, with $L=K_X^{-r}$, and we obtain the corresponding path $A(t)$. 

\

\textbf{Case I}: Every $Z\in \mathcal C$ is a non-trivial K\"ahler-Ricci soliton, i.e. $V_Z\neq 0$. This is our main interest in this paper -- the other case is easier and will be treated later.  

\begin{prop} \label{prop3-2}
Property \textbf{(H1)} holds in this case.
\end{prop}

We need to first determine the element $\Lambda$.
 Given $Z\in \mathcal C$, the vector field $JV_Z$ generates a real one-parameter group of holomorphic isometric actions of $Z$. It induces naturally a one-parameter subgroup $\chi(t)$ of the group $\mathbb U$ of unitary transformations of $H^0(Z, K_Z^{-r})^*$ (with respect to the natural $L^2$-Hermitian inner product defined by $h_Z$). Taking the closure of $\chi(t)$, we obtain a torus $T(Z)\subset \mathbb U$. Then under the natural embedding of $Z\subset \P(H^0(Z, K_Z^{-r})^*)$, the action of $T(Z)$ keeps $Z$ invariant. 
 
  We have a weight space decomposition 
\begin{equation} \label{eqn3-3}
H^0(Z, K_Z^{-r})^*=\bigoplus_{\lambda\in \R} H^0_\lambda(Z, K_Z^{-r})^* 
\end{equation}
such that the $V_Z$-action on $H^0_\lambda(Z, K_Z^{-r})^*$ is given by multiplication by $\lambda$. Clearly different weight spaces are $L^2$-orthogonal. We list the non-trivial weights as $\lambda_1>\lambda_2>\cdots$, and choose an $L^2$-orthonormal basis of $H^0_{\lambda_i}(Z,  K_Z^{-r})$ for each $i$.  Then we put these together in an order that the weights are decreasing, and form an orthonormal basis of $H^0(Z, K_Z^{-r})^*$. For simplicity we call such a basis \emph{compatible}. 

Given a compatible basis we can identify $H^0(Z, K_Z^{-r})^*$ with $\P^{m-1}$. Then we can view the torus $T(Z)$ as a subgroup of $\mathbb T$, the diagonal maximal torus in $U(m)$,  and $JV_Z$ as an element in $Lie(T)=\R^k \subset \R^N$. Notice these do not depend on the choice of a compatible basis. 

\begin{lem}
The map $\mathcal V: \mathcal C\rightarrow \R^N$ sending $Z$ to $JV_Z$ is continuous. In particular, the image of $\mathcal V$ is compact and connected. 
\end{lem}

\begin{proof} Suppose we have a sequence $Z_i\in \mathcal C $ converging to $Z_\infty$. Let $\{s_{i\alpha}\}$ be a compatible basis of $H^0(Z_i, K_{Z_i}^{-r})$.  Passing to a subsequence we may assume $\{s_{i\alpha}\}$ converges to an orthonormal basis $\{s_{\infty\alpha}\}$ of $H^0(Z_\infty, K_{Z_\infty}^{-r})$, under the polarized Gromov-Hausdorff convergence. Given any smooth point $p\in Z_\infty$, we may view the convergence in a neighborhood of $p$ as the smooth convergence of the metric tensors $\omega_{Z_i}$ to $\omega_{Z_\infty}$ on a fixed ball $B\subset \C^n$.  Writing $Ric(\omega_{Z_i})=\omega_{Z_i}+i\p\bp h_i$ and using standard elliptic estimates we may assume $h_i$ converges smoothly to $h_\infty$ on the half ball $B/2$.  Therefore $V_{Z_i}=\nabla_{\omega_{Z_i}} h_i$ converges smoothly to $V_{Z_\infty}$ on any compact subsets of $Z_\infty$. Suppose $\L_{V_{Z_i}}s_{i\alpha}=\mu_{i, \alpha}s_{i\alpha}$, then passing to a subsequence $\mu_{i, \alpha}$ converges to a limit $\mu_{\infty,\alpha}$ and $\L_{V_{Z_\infty}}s_{\infty\alpha}=\mu_{\infty, \alpha}s_{\infty\alpha}$ over $Z_\infty^s$. Since $Z_\infty^s$ is normal we know $s_{\infty\alpha}\in H^0_{\mu_{\infty, \alpha}}(Z_\infty, K_{Z_\infty}^{-r})$. This shows that $\{s_{\infty\alpha}\}$ is a compatible basis, and the map $\mathcal V$ is continuous. 
\end{proof}

\begin{lem} \label{lem3-1}
 There is a unique element $\xi\in \R^N$, such that $JV_Z=\xi$ for all $Z\in \mathcal C$. 
 \end{lem}

\begin{proof}
It suffices to show the image of $\mathcal V$ is a countable set.  To see this we notice that there are countably many subtori of $\T$, and for a given subtorus $T'$ of $\mathbb T$,
 the fixed point set  $\Hilb^{T'}$ of $T'$-action on $\Hilb$ is a projective subscheme so has finitely many connected components.  Thus we only need to show that for a given $T'$, and a connected component $\widetilde \Hilb$ of $\Hilb^{T'}$,  for all $Z\in \mathcal C$ with $T(Z)=T'$ and $[Z]\in \widetilde \Hilb$, $JV_Z$ gives rise to the same element in $Lie(T')$. By \cite{BW} we know for given $Z$, $V_Z$ is characterized as the unique vector in $\R^k$ that satisfies
\begin{equation} \label{eqn3-2}
Fut_{JV_Z}(V')=-\left. \lim_{k\rightarrow\infty}(mk)^{-n-1}\frac{d}{dt} \right|_{t=0}  \left.  Tr \left(e^{JV_Z+tV'} \right) \right|_{H^0(Z, -mkK_Z)} =0
\end{equation}
for all $V'\in Lie(T')$. 
Now since $\widetilde\Hilb$ is connected, the weight decomposition of $H^0(Z, -mrK_Z)$ with respect to $T'$ is the same for all  $Z\in\widetilde\Hilb$, so is the equation (\ref{eqn3-2}). In particular, $JV_Z\in Lie(T')$ is also independent of $Z$.
  \end{proof}

 We define $\Lambda\in \sqrt{-1}Lie(U(m))$ to be the linear transformation of $\C^m$ corresponding to $-J\xi$. For any $Z\in \mathcal C$, under the identification of $H^0(Z, K_Z^{-r})$ with $\C^m$ using a compatible basis, $\Lambda$ coincides with the natural action of $V_Z$ on $H^0(Z, K_Z^{-r})^*$. For simplicity we may also simply view $\Lambda$ as the holomorphic vector field on $Z$. 
 
\

An important ingredient in the proof of Proposition \ref{prop3-2} is the convergence of polarized K\"ahler-Ricci flows proved in \cite{CW2}, which we recall.   For any sequence $t_i\rightarrow\infty$, by \cite{CW2},  passing to a subsequence, $(X, \omega(t_i), h(t_i))$ converges to some $(Z, \omega_Z, h_Z)$ in the polarized Gromov-Hausdorff topology. We can fix a metric on the disjoint union $\bigcup_i(X, \omega(t_i))\cup (Z, \omega_Z)$ which realizes this convergence. 
By Theorem 6 of~\cite{CW2} and the normalization condition (\ref{eqn3-0}), we may assume $h(t_i+\tau)$ and $\omega(t_i+\tau)$ converges smoothly (as tensors) to $\Phi_\tau^*h_Z$ and $\Phi_\tau^*\omega_Z$  over compact subsets of $Z^s$,  uniformly for $\tau \in [0, 2]$. Here $\Phi_\tau$ is the one parameter group of holomorphic transformations generated by $V_Z$.  The key point is that the gauge transformation involved in the process of convergence is chosen uniformly for all $s$.

\begin{proof}[Proof of Proposition \ref{prop3-2}]

We use the initial orthonormal basis $\{s^\alpha\}$ with respect to $H_0$ to identify $E$ with $\C^m$, and hence $G$ with $GL(m;\C)$ and $K$ with $U(m)$. 
Then we adopt  the notations of Section 3.1, and denote $A_i(\tau)=A_{t_i+\tau}$. Then we have 
$$(\dot{A}_i(\tau)A_i(\tau)^{-1})_{\alpha\beta}=-\frac{1}{2}\dot H_{t_i+\tau}(s_\beta(t_i+\tau), s_\alpha(t_i+\tau)).$$
 The right hand side is given by
 $$\int_{X}\langle s_\alpha(t_i+\tau), s_\beta(t_i+\tau) \rangle (-r\dot \phi(t_i+\tau)+\Delta\dot\phi(t_i+\tau))\omega_{\phi(t_i+\tau)}^n. $$
By Perelman's estimate, which was written down by Tian and Sesum in~\cite{SeT} 
and improved by Phong-Sesum-Sturm in~\cite{PSS},
we know $|\dot\phi(t)|$ and $|\Delta\dot\phi(t)|$ are uniformly bounded independent of $t$.  
Therefore $A_i$ is uniformly Lipschitz in $\tau$, so by passing to a subsequence we may assume $A_i$ converges to a Lipschitz map $A_\infty$ from $[0, 2]$ to $G$.  From the definition of polarized convergence, we may also assume $\{A_{t_i}. s_\alpha\}$ converges to an orthonormal basis $\{s_\alpha(\infty)\}$ of $H^0(Z, K_Z^{-r})$ with respect to $h_Z$.  Then for $\tau\in [0, 2]$, $\{A_i(\tau). s_\alpha(t_i)\}$ converges to an orthonormal basis of $H^0(Z, K_Z^{-r})$ with respect to $\Phi_\tau^*h_Z$.  Now we can find an element $h\in U(m)$ such that $\{h. s_\alpha(\infty)\}$ is a compatible orthonormal basis of $H^0(Z, K_Z^{-r})$. Then we easily see $\{e^{\tau\Lambda}h. s_\alpha(\infty)\}$ is an orthonormal basis of $H^0(Z, K_Z^{-r})$ with respect to $\Phi_\tau^*h_Z$. So we have $A_\infty(\tau)=g(\tau)e^{\tau\Lambda}h$ for some $g(\tau)\in U(m)$.  Using the fact that $\dot{A}_\infty(\tau) A_\infty^{-1}(\tau)$  is Hermitian symmetric, we see that $g$ is independent of $\tau$. Therefore $A_i(\tau)A(t_i)^{-1}$ converges uniformly to $ge^{\tau\Lambda}g^{-1}$. 
This proves that $A(t)$ satisfies $(**)$. 

\end{proof}

From the above discussion it follows that $Lim(X)$ is exactly given by the union of the $K_\Lambda$-orbits of $[Z]$ (the image of $Z$ under the embedding using a compatible basis of $H^0(Z, K_Z^{-r})$). 
We now claim $[X]$ satisfies property (\textbf{R}). Indeed for any $Z\in \mathcal C$, the stabilizer group of $[Z]$ in $G_\Lambda$ is isomorphic to $\Aut(Z, V_Z)$. The latter is reductive by Theorem 1.6 in \cite{BW}.
Moreover, suppose $[Z]$ and $[Z']$ are in the same $G_\Lambda$-orbit, then $Z$ and $Z'$ are isomorphic as $\mathbb Q$-Fano varieties, and by Theorem 1.4 in \cite{BW}, $Z$ and $Z'$ are indeed the same point in $\mathcal C$. Then it follows from Theorem \ref{thm2-12} that $Lim(X)=K_\Lambda. [X_\infty]$ for a single $[X_\infty]$.  This shows that for all $Z\in \mathcal C$, the underlying $\Q$-Fano variety $(Z, V_Z)$ is isomorphic to $X_\infty$. Then by Theorem 1.4 of \cite{CW2} the corresponding weak K\"ahler-Ricci soliton metric is also unique up to the action of $\Aut(Z, V_Z)$.  This is the precise meaning of the uniqueness statement in Theorem \ref{thm1-3}. Notice  the definition of polarized Gromov-Hausdorff limit in \cite{DS1} also involves a limit connection on $K_Z^{-r}|_{Z^s}$. This is irrelevant for our purpose in this paper and we leave this for future work. 

The one parameter subgroup $\phi:\C^*\rightarrow G$ constructed from Theorem \ref{thm2-12} gives rise to  a $\Lambda$-equivariant test configuration for $\overline X$ with central fiber $X_\infty$. By the openness of normality in a flat family (see for example \cite{GT}, appendix E), we conclude that $\overline X$ is also normal. 

\begin{prop} \label{prop3-5}
In this case,  $X$ is  K-unstable. 
\end{prop}

\begin{proof}
Similar to the proof of Lemma \ref{lem3-1}, we know
$$Fut(\overline X, \Lambda)=Fut(X_\infty, \Lambda). $$
Here the notation means the usual Futaki invariant computed using the holomorphic vector field $\Lambda$. Notice $\Lambda=V_{X_\infty}$. By the discussion of Section 3 in \cite{BW} which generalizes the result of Tian-Zhu \cite{TZ} to the case of $\Q$-Fano varieties, we know on the space $\t$ of holomorphic vector fields on $X_\infty$ that commute with $V_{X_\infty}$, there is a strictly convex function $F$, such that for any $V, W\in \t$, 
$$Fut_V(X_\infty, W)= \left. \frac{d}{dt} F(V+tW)\right|_{t=0}. $$
Note that $Fut_V$ is the same one as defined in~\cite{BW} and~\cite{TZ}. 
Since $(X_\infty, V_\infty)$ is a weak K\"ahler-Ricci soliton, $V_\infty$ is a critical point of $F$. This implies 
$$Fut(X_\infty, \Lambda)=Fut_0(X_\infty, V_\infty)<Fut_{V_\infty}(X_\infty, V_\infty)=0.$$

 Now as in Section 3.1 we choose a rational $\Gamma\in Lie(T)$ that is sufficiently close to $\Lambda$, and obtain a test configuration for $X$ with central fiber $\overline X$. Since the Futaki invariant depends linearly on the holomorphic vector field, we can assume $Fut(\overline X, \Gamma)<0$. Hence $X$ is K-unstable. 
\end{proof}

\begin{prop} \label{prop3-6}
$(X, h(t))$ satisfies Property \textbf{(H2)}.
\end{prop}

\begin{proof}
We have the natural map $\iota_{k, t}: \text{Sym}^k H^0(X, L)\rightarrow H^0(X, L^k)$, where both spaces are endowed with the $L^2$-metric $H_{t}$. Given any sequence $t_i\rightarrow\infty$, by passing to a subsequence we may assume $(H^0(X, L^k), H_{t_i})$ converges naturally to $(H^0(Z, -K_Z^{kr}), H_Z)$, where $H_Z$ is the $L^2$ inner-product defined by $h_Z$, and $\iota_{k, t_i}$ converges to $\iota_{k, \infty}$, which is also surjective. The conclusion follows from this and the definition of $H_t^*$. 
\end{proof}

In particular, by the discussion of Section 3.1, the filtration $\F$ (and hence $(\overline X, \Lambda)$) are intrinsically defined by $(X, h(t))$. 

\

\textbf{Case II}. There is one limit in $\mathcal C$ which is K\"ahler-Einstein. In this case, it is not hard to see that every $Z\in \mathcal C$ is K\"ahler-Einstein.

Actually, by the monotonicity of Perelman's $\mu$-functional along the K\"ahler-Ricci flow, we know that every limit in $\mathcal{C}$ has the same $\mu$-functional level $\mu_{\infty}$, which is the
one of the K\"ahler-Einstein metric.  Suppose $X_{\infty} \in \mathcal{C}$ and $f_{\infty}$ is the Ricci potential on $X_{\infty}$.  By soliton equation on regular part of $X_{\infty}$, we have
\begin{align*}
   &R+\Delta f_{\infty}-n=0, \\
   &R+2 \Delta f_{\infty} -|\nabla f_{\infty}|^2 +f_{\infty}-2n=\mu_{\infty}.
\end{align*}
Since $\mu_{\infty}$ is the $\mu$-functional level of weak K\"ahler-Einstein metric,  it is easy to see that $\mu_{\infty}=-n+\log \frac{Vol(X)}{(2\pi)^n}$.
Combining this with the above equations, we obtain
\begin{align*}
   \mu(X_{\infty})&=\int_{X_{\infty}} \left\{ R+|\nabla f_{\infty}|^2 +f_{\infty}-2n \right\}  (2\pi)^{-n} e^{-f_{\infty}} \\
      &=\int_{X_{\infty}} \left\{ R+\Delta f_{\infty} +f_{\infty}-2n \right\}  (2\pi)^{-n} e^{-f_{\infty}}\\
      &=-n+ \log \frac{Vol(X)}{(2\pi)^n}. 
\end{align*}
Note that the integration by parts work here, due to the high codimension of singularity (Minkowski codimension strictly greater than $2$, see section 2 of \cite{CW2} for more details) of $X_{\infty}$ and the uniform boundedness of $f_{\infty}$.  It follows that
\begin{align*}
  (2\pi)^{-n} \int_{X_{\infty}} f_{\infty} e^{-f_{\infty}}= -n\log 2\pi +\log Vol(X). 
\end{align*}
By Jensen inequality for the convex function $x\log x$, the above equality implies that $f_{\infty}=constant=\log \frac{Vol(X)}{(2\pi)^n}$.  Consequently, $X_{\infty}$ must be weak K\"ahler-Einstein.

Then by arguments similar to the proof of Theorem \ref{thm2-12} (indeed easier since there is no $\Lambda$ involved),  one can prove the uniqueness of the limits in $\mathcal C$.   It also follows from \cite{CDS3} that  in this case $X$ is always K-semistable, and if $X$ is K-stable, then $X_\infty=X$, and $X$ admits a K\"ahler-Einstein metric.

\

To summarize, we have proved the following

\begin{itemize}
\item If $X$ is K-unstable, then the flow converges to a unique $\Q$-Fano variety $X_\infty$ endowed with a non-trivial weak K\"ahler-Ricci soliton metric. 
\item If $X$ is K-stable, then the flow converges to a unique K\"ahler-Einstien metric on $X$;
\item If $X$ is K-semistable but not K-stable, then the flow converges to a unique $\Q$-Fano variety $X_\infty$ endowed with a weak K\"ahler-Einstein metric. 
\end{itemize}

Theorem \ref{thm1-3} follows directly from this.  Proposition \ref{prop3-6} motivates us the following

\begin{conj} \label{conj3-7}
When $X$ is K-unstable, the geometric objects $\overline X$, $\mathcal F$, $X_\infty$ are uniquely determined by $X$. 
In other words, these are independent of the choice of the initial metric $\omega_0$. 
\end{conj}

This is also related to the work of Darvas-He \cite{DH}, where it is shown that the K\"ahler-Ricci flow trajectories give rise to non-trivial geodesic rays in the space of K\"ahler potentials.  
Suppose the conjecture is true, then we can ask a sensible algebro-geometric question

\begin{prob} \label{prob3-8}
Determine $\overline X$, $\mathcal F$ and $X_\infty$ in terms of the algebraic geometry of $X$, for example, the filtration should maximize an appropriate notion of ``normalized Futaki invariant". 
\end{prob}

\

Now we prove Corollary \ref{cor1-5}. Suppose we are given a Fano manifold $X$ together with a holomorphic vector field $V$ such that $JV$ generates a compact subgroup $H$ of $\Aut(X)$, and then we can assume the initial metric $h(0)$ and $\omega(0)$ are $H$ invariant. It follows that $V$ also induces a natural holomorphic vector field on $\overline X$ and $X_\infty$. For simplicity of notation we also denote this by $V$. 
Then by the definition of relative K-stability, if $(\overline X, \Lambda)$ is not isomorphic to $(X, V)$, then $Fut_{V}(\overline X, \Gamma)>0$ for all rational $\Gamma\in Lie(T)$ close to $\Lambda$. This implies that $Fut_{V}(\overline X, \Lambda)\geq 0$. On the other hand,  we can write 
 $$Fut_{V}(X_\infty, \Lambda)=Fut_{V}(X_\infty, \Lambda-V)+Fut_V(X_\infty, V)$$
Since $(X_\infty, \Lambda)$ is K\"ahler-Ricci soliton and $[\Lambda, V]=0$,  by the results of \cite{BW} and similar discussion as above, we know
$$Fut_V(X_\infty, \Lambda-V)\leq Fut_{\Lambda}(X_\infty, \Lambda-V)=0 $$
with equality if and only if $V=\Lambda$.
Applying the relative K-stability of $(X, V)$ to the product test configurations, we see $Fut_V(X, V)=0$. Now we know $X$, $\overline X$ and $X_\infty$ all lie in the same component of the subscheme of $\Hilb$ fixed by $V$, so as in the proof of Lemma \ref{lem3-1} we know $Fut_{V}(X_\infty, V)$ vanishes as well. Therefore we conclude that $Fut_{V}(X_\infty, \Lambda)\leq 0$, and hence $V=\Lambda$. 
This  implies that $(\overline X, \Lambda)$ is isomorphic to $(X, V)$. Then using relative K-stability again we conclude that $(X_\infty, \Lambda)$ is also isomorphic to $(X, V)$, which shows the existence of K\"ahler-Ricci soliton on $(X, V)$. 

\begin{rmk}
Suppose $X$ is endowed with an action of a compact group $H$, then the above arguments can also be used to show that $X$ admits a K\"ahler-Einstein metric if and only if $X$ is $H$-equivariantly stable. This has been proved by \cite{DaSz}, using the classical continuity path.  It is further observed in \cite{DaSz} that the ``equivariant K-stability" is sometimes verifiable for manifolds with large symmetry, including toric Fano manifolds and Fano threefolds with an action of a two dimensional torus.
\end{rmk}

\begin{rmk}
For the limit K\"ahler-Ricci soliton $X_\infty$, it has been proved in \cite{CW2} that the smooth part of any tangent cone is Ricci-flat. Using this the results of \cite{DS2} can be extended to our case and we leave this for future work. \end{rmk}

\section{The Calabi flow and stability}

Suppose $(X, L)$ is a polarized K\"ahler manifold, starting from any metric $\omega(0)\in 2\pi c_1(L)$,  the Calabi flow $\omega(t)=\omega(0)+i\p\bp \phi(t)$ is a fourth order parabolic equation on $\phi(t)$, given by
\begin{equation} \label{eqn1-2}
\frac{\p \phi(t)}{\p t}=S(\omega(t))-\underline{S} 
\end{equation}
where $S(\omega_t)$ is the scalar curvature of $\omega_t$, and $\underline{S}$ is the average of $S(\omega_t)$ (independent of $t$). 
This is a promising approach to tackle the Yau-Tian-Danaldson conjecture relating existence of extremal K\"ahler metrics in $2\pi c_1(L)$ and K-stability of $(X, L)$. 
We will not discuss the analytic aspects of the Calabi flow, which has seen significant progress recently. 
For instance, one can check the work of Chen-He~\cite{CH1},Tosatti-Weinkove~\cite{ToWe}, He~\cite{He}, Streets~\cite{St}, Huang-Feng~\cite{HF}, Li-Wang-Zheng~\cite{LWZ} and the references therein for more information of the recent development. 
However,  in the current paper,  our  focus here is again on the relation with K-stability.  In particular, we will prove

\begin{thm} \label{thm1-6}
Given a smooth solution $\omega(t)(t\in [0, \infty))$ of  (\ref{eqn1-2}). Suppose $\omega(t)$ has uniformly bounded curvature and diameter, then

\begin{enumerate}
\item $(X, L, \omega(t))$ converges to a unique limit $(X', L', \omega')$ in the sense of Cheeger-Gromov, where $\omega'\in 2\pi c_1(L')$ is an extremal K\"ahler metric, i.e. $\nabla^{1, 0}_{\omega'}S(\omega')$ is a holomorphic vector field. 
\item  If $X$ is  K-stable, then $(X', L')$ is isomorphic to $(X, L)$ and $\omega'$ has constant scalar curvature. In particular, $(X, L)$ admits a constant scalar curvature K\"ahler metric.
\item  If $X$ is  strictly K-semistable, then $\omega'$ has constant scalar curvature, and there is a test configuration for $(X, L)$ with central fiber $(X', L')$.
\item  If $X$ is K-unstable, then there is a test configuration $\mathcal X$ for $(X, L)$ with central fiber $(X'', L'')$, which is naturally associated to $\omega(t)$, with 
$$Fut(\mathcal X)/N_2(\mathcal X)=-\inf_{\omega\in 2\pi c_1(L)} ||S(\omega)-\underline{S}||_{L^2}. $$
Here $N_2(\mathcal X)$ is the norm defined in \cite{Do06}.  In particular, in view of \cite{Do06}, $\mathcal X$ is an optimal test configuration with minimal Futaki invariant. 
\end{enumerate}
\end{thm}

The proof of this is similar to that of Theorem \ref{thm1-3}, but simpler.  We now briefly sketch the main arguments. 
We fix a Hermitian metric $h_0$ on $L$ with curvature $-\sqrt{-1} \omega(0)$.  Then $h(t)=h(0)e^{-\phi(t)}$ has curvature $-\sqrt{-1}\omega(t)$. 
By our assumption we may obtain \emph{polarized} Cheeger-Gromov compactness.  Namely, given any sequence $t_i\rightarrow\infty$, passing to subsequence, we may obtain a polarized limit $(X', L', \omega', h')$. 
We claim $\omega'$ is an extremal K\"ahler metric. This follows from well-known arguments. Recall the Calabi functional is defined as 
$$Ca(\omega)=\int (S(\omega)-\underline S)^2\omega^n.$$
Direct calculation (see for example Chen-He \cite{CH1}) shows that
\begin{equation} \label{eqn4-2}
\frac{d}{dt} Ca(\omega(t))=-\int |\bp \nabla_tS(\omega(t))|^2\omega(t)^n\leq 0.
\end{equation}
This in particular implies that $Ca(\phi(t_i-1))-Ca(\phi(t_i+1))$ converges uniformly to zero. 
By the parabolic curvature estimates in \cite{CH2} we know the path 
$$\left\{\phi_i(t)=\phi(t_i+t)-\phi(t_i), t\in [-1, 1] \right\}$$ 
converges smoothly (with respect to both the time and space variables) to a path $\phi_\infty(t)$ and $\omega'(t)=\omega'+i\p\bp\phi_\infty(t)$ also solves the Calabi flow equation. Now since 
$$Ca(\phi(t_i-1))-Ca(\phi(t_i+1))=\int_{t_i-1}^{t_i+1} |\bp \nabla_tS(\omega(t))|^2\omega(t)^ndt$$
we easily conclude that $\bp \nabla S(\omega')=0$, i.e.,  $\omega'$ is an extremal K\"ahler metric. This proves (1), except the uniqueness.

Then as in Section 3.2 we let $\mathcal C$ be the set of all such sequential limits, and $\overline{\mathcal C}$ be the union of $\mathcal C$ and $\{(X, L, \omega(t))|t\geq 0\}$. It is then easy to find $r$ and $m$ depending only on $(X, L, \omega(0))$ such that any $Z\in \overline {\mathcal C}$ is holomorphically embedded into $\P^{m-1}$ by $L^2$-orthonormal sections of $L^r$, and the image lies in a fixed Hilbert scheme $\Hilb$, and the natural map $\overline {\mathcal C}\rightarrow\Hilb/U(m)$ is continuous.  Then we can apply the discussion of Section 3.1, with $L$ replaced by $L^r$, and get the path $A(t)$.

Again as in Section 3.2, it suffices to deal with the case that none of the limits in $\mathcal C$ has constant scalar curvature (the other case is easier, and is already treated in \cite{CS}),
 so we will always assume this is case in the remainder of this subsection.  

\begin{prop} \label{prop4-2}
Property \textbf{(H1)} holds. 
\end{prop}

\begin{proof}
 The key property is that by \cite{FM}, the extremal vector field $V'=\nabla' S(\omega')$ is always rational, i.e. it always generates an $S^1$-action on $X'$ which also lifts to $L'$. This implies an analogous statement to Lemma \ref{lem3-1} is true, and the proof is simpler (since there are at most countably many rational elements in the an abelian Lie algebra). From here the proof of our claim is exactly the same as Proposition \ref{prop3-2}. 
\end{proof}

We first obtain that there is a unique element $(X_\infty, L_\infty)$ in $\mathcal C$, by Calabi's structure theorem for extremal metrics (which says that $\Aut(X', L', V')$ is reductive) and the uniqueness of extremal K\"ahler metrics on a fixed polarized K\"ahler manifold. Then by similar arguments to Proposition \ref{prop3-5} we know $(X, L)$ is K-unstable and there is a test configuration $\mathcal X$ for $(X, L^r)$ with central fiber $(\overline X, \mathcal O(1)|_{\overline X})$ with negative Futaki invariant. Again the proof is simpler since our assumptions rule out the appearance of possible singularities and by rationality of the extremal vector field we do not need to perturb the $\Lambda$. Moreover, as Proposition \ref{prop3-6} we know $(X, h(t))$ satisfies \textbf{(H2)}, and so we obtain an intrinsic description of $(\overline X, \mathcal O(1)|_{\overline X})$ as the scheme corresponding to the graded ring associated to the filtration of $\bigoplus_{k\geq 0} H^0(X, L^{rk})$ defined by the Calabi flow solution $\omega(t)$.  Furthermore, $\overline X$ is smooth since by Theorem \ref{thm2-12} there is a $\Lambda$-equivariant test configuration for $\left(\overline X, \mathcal O(1)|_{\overline X} \right)$ with central fiber $(X_\infty, L_\infty)$, and smoothness is an open condition among a flat family. 

Now it remains to prove the last statement in Theorem \ref{thm1-6}. 
For this we notice by definition 
$$Fut(\mathcal X)=Fut(X_\infty, \Lambda)=-||S(\omega')-\underline S||_{L^2}^2.$$
By the smooth convergence we have 
$$||S(\omega')-\underline S||_{L^2}^2 \geq  \inf_{\omega\in 2\pi c_1(L)} ||S(\omega)-\underline S||_{L^2}^2. $$
By definition in \cite{Do06} and equivariant Riemann-Roch theorem,  we have 
$$N_2(\mathcal X)^2=||S(\omega')-\underline S||_{L^2}^2$$
Hence we have 
$$Fut(\mathcal X)/N_2(\mathcal X)\leq -\inf_{\omega\in 2\pi c_1(L)} ||S(\omega)-\underline S||_{L^2}^2.$$
On the other hand, by \cite{Do06}, we also have 
$$||\mathcal X||^{-1}Fut(\mathcal X)\geq -\inf_{\omega\in 2\pi c_1(L)} ||S(\omega)-\underline S||_{L^2}. $$
Therefore the inequality holds. This finishes the proof of Theorem \ref{thm1-6}. 

\

There are a few remarks. 

\begin{enumerate}[(1)]

\item One can also formulate conjectures relating the above $\overline X$ with optimal degeneration, similar to Conjecture \ref{conj3-7} and Problem \ref{prob3-8}.  The difference is that here we can use the known notion of an ``optimal degeneration", as introduced in \cite{Do06}.

\item It seems also possible to allow suitable classes of singularities to occur, as in the case of K\"ahler-Ricci flow on Fano manifolds. This, together with an appropriate weak compactness theory,  might lead to a proof of the Yau-Tian-Donaldson conjecture in some special cases. We will leave this for future work. 
\end{enumerate}

\vspace{0.5in}

Xiuxiong Chen, Department of Mathematics,  Stony Brook University,
NY, 11794, USA;
School of Mathematics, University of Science and Technology of China, Hefei, Anhui, 230026, PR China;
xiu@math.sunysb.edu.\\

Song Sun, Department of Mathematics, Stony Brook University,
NY, 11794, USA; song.sun@stonybrook.edu.\\

Bing  Wang, Department of Mathematics, University of Wisconsin-Madison,
Madison, WI, 53706, USA;  bwang@math.wisc.edu.\\

\end{document}